\documentclass[12pt]{amsart}
\usepackage{amsfonts,amsthm,amssymb,amsmath}
\usepackage{hyperref}

\def\Z{{\mathbb Z}}
\def\Qq{{\mathbb Q}}
\def\R{{\mathbb R}}
\def\C{{\mathbb C}}
\def\H{{\mathfrak H}}

\def\sq{\hbox{\rlap{$\sqcap$}$\sqcup$}}
\def\qed{\ifmmode\sq\else{\unskip\nobreak\hfil
         \penalty50\hskip1em\null\nobreak\hfil\sq
         \parfillskip=0pt\finalhyphendemerits=0\endgraf}\fi}

\def\smat#1#2#3#4{\left(\begin{smallmatrix}#1&#2\\#3&#4\end{smallmatrix}\right)}

\newtheorem{theorem}{Theorem}
\newtheorem{lemma}[theorem]{Lemma}

\newtheorem{cor}[theorem]{Corollary}

\numberwithin{theorem}{section}
\numberwithin{equation}{section}

\title[On the spinor $L$-function of Miyawaki-Ikeda lifts]{On the spinor $L$-function of Miyawaki-Ikeda lifts}
\author[S.~Hayashida]{Shuichi Hayashida
}
\date{\today}

\begin{document}

\begin{abstract}
We consider lifts from two elliptic modular forms to Siegel modular forms
of odd degrees which are special cases of Miyawaki-Ikeda lifts.
Assuming non-vanishing of these Miyawaki-Ikeda lifts, we show that
the spinor $L$-functions of these Miyawaki-Ikeda lifts
are products of some kind of symmetric power $L$-functions determined by original two
elliptic modular forms. 
\end{abstract}

\maketitle

\section{Introduction}\label{sec:intro}

\subsection{}
In \cite{Mi} Miyawaki conjectured the existence of lifts
from pairs of two elliptic modular forms of level one
to Siegel modular forms of degree three.
This Miyawaki's conjecture was partially solved by Ikeda~\cite{Ik2}.
Furthermore, Ikeda constructed certain lifts from pairs of an elliptic modular form
and a Siegel modular form to Siegel modular forms in general degrees.
In this paper we call these lifts \textit{Miyawaki-Ikeda lifts}.



In \cite{He} Heim showed that the \textit{spinor $L$-function} of the Siegel cusp form
of weight $12$ of degree $3$, which is a Miyawaki-Ikeda lift, is a product of certain $L$-functions
obtained from two elliptic cusp forms.
This identity of the spinor $L$-function was originally conjectured by Miyawaki~\cite{Mi}.
%
The purpose of the present paper is to generalize the result
about the spinor $L$-function in~\cite{He}
for Miyawaki-Ikeda lifts in arbitrary weights and in arbitrary
odd degrees.

We remark that the identity for the \textit{standard $L$-function} of Miyawaki-Ikeda lifts
have already been given in~\cite{Ik2}.
For the case of Miyawaki-Ikeda lifts of two elliptic modular forms,
in~\cite{MaassRe} we obtain another proof of this identity of the standard $L$-function
by using generalized Maass relations.

\subsection{}
We explain our results more precisely.
Let $f \in S_{2k}(\Gamma_1)$ and $g \in S_{k+n}(\Gamma_1)$ be elliptic modular forms
of weight $2k$ and of weight $k+n \in 2\Z$, respectively.
We assume that both forms $f$ and $g$ are normalized Hecke eigenforms.
Let $\lambda_f(p)$ and $\lambda_g(p)$ be the $p$-th Fourier coefficients
of $f$ and $g$, respectively.
We denote by $\{\alpha_p^{\pm}\}$ and $\{\beta_p^{\pm}\}$ the sets of complex numbers
determined through the identities:
\begin{eqnarray*}
 1 - \lambda_f(p)\, p^{-s} + p^{2k-1-2s}
 &=&
 \left( 1 - \alpha_p\, p^{k-\frac12-s} \right)
 \left( 1 - \alpha_p\,^{-1} p^{k-\frac12-s} \right), \\
 1 - \lambda_g(p)\, p^{-s} + p^{k+n-1-2s}
 &=&
 \left( 1 - \beta_p\, p^{\frac{k+n-1}{2} - s} \right)
 \left( 1 - \beta_p^{-1}\, p^{\frac{k+n-1}{2} - s} \right).
\end{eqnarray*}
For any integer $m > 1$ we set
\begin{eqnarray*}
 A_{p,m-1} := \mbox{diag}(\alpha_p^{m-1},\alpha_p^{m-3},...,\alpha_p^{-m+1})
 \mbox{ and }
 B_p := \smat{\beta_p}{}{}{\beta_p^{-1}},
\end{eqnarray*}
and define
\begin{eqnarray*}
  \qquad \quad
  L(s,g \otimes \mbox{sym}_{m-1}\, f) := 
  \prod_p \det(1_{2m} - A_{p,m-1} \otimes B_p \cdot p^{(m-1)(k-\frac12) + \frac{k+n-1}{2} - s})^{-1}
\end{eqnarray*}
for $\mbox{Re}(s) > (m-1)(k-\frac12) + \frac{k+n-1}{2} + 1$.
We put $L(s,g \otimes \mbox{sym}_0\, f) := L(s,g)$, which is the usual Hecke $L$-function of $g$.

Let $\mathcal{F}_{f,g} \in S_{k+n}(\Gamma_{2n-1})$ be the Siegel modular form
of weight $k+n$ of degree $2n-1$ which is the Miyawaki-Ikeda lift of $(f,g)$.
(see Section~\ref{sec:spinor_Miyawaki} for the definition of $\mathcal{F}_{f,g}$.)
Let $T_{2n-1}(p)$ be the Hecke operator which will be denoted in the next section.
Then we obtain the main theorem.

\begin{theorem}\label{thm:main}
We assume $\mathcal{F}_{f,g} \not \equiv 0$.
Then $\mathcal{F}_{f,g}$ is an eigenform for Hecke operators $T_{2n-1}(p)$
for any prime $p$.
Furthermore, the spinor $L$-function of $\mathcal{F}_{f,g}$ satisfies
\begin{eqnarray*}
 L(s,\mathcal{F}_{f,g},\mbox{spin}) 
 &=&
 L(s,g\otimes \mbox{sym}_{n-1}f  ) \\
 && \times 
 \prod_{m=1}^{n-1}
 \prod_{\begin{smallmatrix} r = -m(2n-m-2) \\ (step\ 2) \end{smallmatrix}}^{m(2n-m-2)}
 L(s-m(k-\frac12)+\frac{r}{2},g\otimes \mbox{sym}_{n-m-1}f  )^{\beta(r,m,n-1)}
\end{eqnarray*}
for $\mbox{Re}(s) > (n-\frac12)k + 1$,
where $\beta(r,m,n)$ is a non-negative integer
given in~\cite{Sch} (see also (\ref{id:beta}) in Section~\ref{sec:Ikeda}.)
\end{theorem}

In the case of $(n,k) = (2,10)$, namely in the case of $\mathcal{F}_{f,g} \in S_{12}(\Gamma_3)$,
this theorem has been shown by Heim \cite{He}.

We remark that Theorem~\ref{thm:main} has already been given in \cite{Mu}.
However, the proof in~\cite{Mu} is not available,
because a necessary property for the Fourier coefficients of
the Klingen-Eisenstein series is not known
(see also the Final remarks in \cite{He}.)
The proof in the present paper for Theorem~\ref{thm:main}
depends on a generalization of Maass relation for Duke-Imamoglu-Ibukiyama-Ikeda lifts.


\vspace{5mm}

This paper is organized as follows:
In Section~\ref{sec:notation} we will give notation and in Section~\ref{sec:Ikeda}, we will review the spinor $L$-function
of the Duke-Imamoglu-Ibukiyama-Ikeda lifts.
In Section~\ref{sec:spinor_Miyawaki} we will review the Miyawaki-Ikeda lifts and we shall give the proof of Theorem~\ref{thm:main}.

\vspace{5mm}

Acknowledgement:
to be entered later.

\section{Notation}  \label{sec:notation}

\subsection{Symbols}
The symbol $\mbox{diag}(a_1,...,a_n)$ denotes the diagonal matrix with diagonal elements $a_1$, ..., $a_n$.
The symbol $M_{n,m}(R)$ is the set of all matrices of size $n$ times $m$ with entries in a commutative ring $R$.
We put $M_{n}(R) := M_{n,n}(R)$.
For a natural number $n$ we set $ <n> := \frac{n(n+1)}{2}$.
The letter $p$ is reserved for prime numbers.

Let $\Gamma_n := \mbox{Sp}_n(\Z)$ be the symplectic group with matrix of size $2n$ with entries in $\Z$.
For the field $K = \Qq$ or $\R$, we define
\begin{eqnarray*}
 \mbox{GSp}_n^+(K)
 &:=&
 \left\{
   M \in GL_{2n}(K) \, ; \, M J_{2n} {^t M} = \nu(M) J_{2n}, \, \nu(M) > 0
 \right\},
\end{eqnarray*}
where $J_{2n} := \smat{0}{1_n}{-1_n}{0} $.
The above number $\nu(M)$ is called the similitude of $M \in \mbox{GSp}_n^+(K)$.

The symbol $\H_n$ denotes the Siegel upper half space of size $n$.

We denote by $S_{k}(\Gamma_{n})$ the space of Siegel cusp forms of weight $k$
of degree $n$ and denote by
$J_{k,m}(\Gamma_{n})$
(resp. $J_{k,m}^{cusp}(\Gamma_{n})$) the space of Jacobi forms (resp. Jacobi cusp forms)
of weight $k$
of index $m$ of degree $n$ (cf. \cite{Zi} for the definition of Jacobi forms in higher degrees.)
We remark that any element in $J_{k,m}(\Gamma_{n})$ is a holomorphic
function on $\H_n \times \C^n$.


\subsection{$L$-functions of Siegel modular forms}
For any function $F$ on $\mathfrak{H}_{n}$ and
for $M = \smat{A}{B}{C}{D} \in  \mbox{GSp}_{n}^+(\R)$ we define the slash-operator $|_k$ by
\begin{eqnarray*}
  (F|_{k} M)(\tau) := \det(C\tau + D)^{-k} F(M\cdot \tau). 
\end{eqnarray*}
Here $M \cdot \tau := (A\tau + B)(C \tau + D)^{-1}$ is the linear fractional transformation.

For any Siegel modular form $F$ of weight $k$ with respect to $\Gamma_n$
and for any positive integer $N$,
we define
\begin{eqnarray*}
 F| T_{n}(N)
 &:=& 
 N^{nk-<n>} \sum_{M \in \Gamma_{n} \backslash S_{n}(N)}
              F|_{k} M.
\end{eqnarray*}
Here we defined
\begin{eqnarray*}
 S_{n}(N)
 &:=&
 \left\{
  M \in \mbox{GSp}_{n}^{+}(\Qq) \cap M_{2n}(\Z) \, | \, \nu(M) = N
 \right\}.
\end{eqnarray*}

We denote by
$\mathcal{H}(\Gamma_{n}, \mbox{GSp}_{n}^+(\Qq))$
(resp. $\mathcal{H}(\Gamma_{n},\mbox{Sp}_{n}(\Qq))$)
the Hecke ring with respect to the Hecke pair
$(\Gamma_{n}, \mbox{GSp}_{n}^+(\Qq))$
(resp. $(\Gamma_{n},\mbox{Sp}_{n}(\Qq))$.)

We set $\Delta_{n,p} := \mbox{GSp}_{n}^+(\Qq)\cap \mbox{GL}_{2n}(\Z[\frac{1}{p}])$.
Then
it is known that the Hecke ring $\mathcal{H}(\Gamma_{n}, \mbox{GSp}_{n}^+(\Qq))$
has the local decomposition
\begin{eqnarray*}
\mathcal{H}(\Gamma_{n}, \mbox{GSp}_{n}^+(\Qq)) \cong
  \bigotimes_p
   \mathcal{H}(\Gamma_{n}, \Delta_{n,p}).
\end{eqnarray*}
There exists the isomorphism
$
\mathcal{H}(\Gamma_{n}, \Delta_{n,p})
\cong
\C[x_0^{\pm}, x_1^{\pm}, ..., x_{n}^{\pm}]^{W_n}
$
which is called the Satake isomorphism,
where $\C[x_0^{\pm}, x_1^{\pm}, ..., x_{n}^{\pm}]^{W_n}$
is the subring of all $W_n$-invariant elements in the polynomial ring
$\C[x_0^{\pm}, x_1^{\pm}, ..., x_{n}^{\pm}]$.
Here $W_n$ is the Weyl group of symplectic group.
The action of $W_n$ on $\C[x_0^{\pm}, x_1^{\pm}, ..., x_{n}^{\pm}]$ is generated by the maps
\begin{eqnarray*}
 \sigma_i \ : \ x_0 \mapsto x_0 x_i,\ x_i \mapsto x_i^{-1},\ x_j \mapsto x_j \ (j \neq i)
\end{eqnarray*}
for $i=1,...,n$, and by all permutations of $\{x_1,...,x_{n} \}$.

%

Let $F \in S_{k}(\Gamma_{n})$ be a common eigenform for all Hecke operators for
$\mathcal{H}(\Gamma_{n}, \mbox{GSp}_{n}^+(\Qq))$.
Then there exists a set of complex numbers $\{\mu_0,\mu_1,...,\mu_{n} \} := \{\mu_{0,p},\mu_{1,p},...,\mu_{n,p} \}$
determined by the eigenvalues of $F$ and by the Satake isomorphism.
Such a set of complex numbers 
$\{\mu_0,\mu_1,...,\mu_{n} \}$
is called the Satake parameters and
is uniquely determined by $F$ and by prime $p$ up to the action of
the Weyl group $W_n$.

The spinor $L$-function $L(s,F,\mbox{spin})$
and the standard $L$-function $L(s,F,\mbox{st})$ are defined by
\begin{eqnarray*}
  L(s,F,\mbox{spin})
  &:=&
  \prod_{p} \left\{ (1-\mu_0 p^{-s}) \prod_{j=1}^{n}\prod_{1\leq i_1 <
   \cdots < i_j \leq n}(1-\mu_0
 \mu_{i_1} \cdots \mu_{i_j} p^{-s}) \right\}^{-1}
\end{eqnarray*}
and by
\begin{eqnarray*}
 L(s,F,\mbox{st})
 &:=&
  \prod_{p} \left\{ (1-p^{-s}) \prod_{i=1}^{n}   \left( 1 - \mu_i p^{-s} \right)
                                                          \left( 1 - \mu_i^{-1} p^{-s} \right) \right\}^{-1},
\end{eqnarray*}
where $\{\mu_0,\mu_1,...,\mu_{n}\}$ is the Satake parameters of $F$.
Here
\begin{eqnarray*}
\mu_0^2 \mu_1 \cdots \mu_{n} = p^{n k - <n>}.
\end{eqnarray*}

To determine the standard $L$-function of $F$, the function $F$
is only required to be a common eigenform for all Hecke operators
for $\mathcal{H}(\Gamma_{n},\mbox{Sp}_{n}(\Qq))$.
If $F$ is a common eigenform for all Hecke operators
for $\mathcal{H}(\Gamma_{n},\mbox{Sp}_{n}(\Qq))$
and also for all Hecke operators $T_{n}(p)$ for any prime $p$,
then $F$ is a common eigenform for all Hecke operators
for $\mathcal{H}(\Gamma_{n},\mbox{GSp}_{n}^+(\Qq))$,
and the spinor $L$-function of $F$ is determined.

\section{The spinor $L$-function of Duke-Imamoglu-Ibukiyama-Ikeda lifts}\label{sec:Ikeda}
In this section we review the identity about the spinor $L$-function of Duke-Imamoglu-Ibukiyama-Ikeda lifts.
This identity has been given by Schmidt \cite{Sch} and Murakawa \cite{Mu}, independently.

Let $f \in S_{2k}(\Gamma_1)$ be a normalized Hecke eigenform.
We put $F_{2n} \in S_{k+n}(\Gamma_{2n})$ the Duke-Imamoglu-Ibukiyama-Ikeda lift of $f$.
We remark that this lift was conjectured by Duke and Imamoglu and also by Ibukiyama,
independently, and shown by Ikeda~\cite{Ik}.
In~\cite{Ik} Ikeda proved that $F_{2n}$ is a common eigenform for all Hecke operators for
$\mathcal{H}(\Gamma_{2n},\mbox{Sp}_{2n}(\Qq))$
with the standard $L$-function
\begin{eqnarray*}
 L(s,F_{2n},st)
 &=&
 \zeta(s)
 \prod_{i=1}^{2n} L(s+k+n-i,f).
\end{eqnarray*}

Murakawa~\cite{Mu} and Schmidt~\cite{Sch} independently obtained the following lemma.

\begin{lemma}[\cite{Mu,Sch}]\label{lem:spin_ms}
The Duke-Imamoglu-Ibukiyama-Ikeda lift $F_{2n}$ is a common eigenform for all Hecke operators for
$\mathcal{H}(\Gamma_{2n}, \mbox{GSp}_{2n}^+(\Qq))$.
Furthermore, the Satake parameters $\{\mu_0, \mu_1 , ... , \mu_{2n} \}$ of $F_{2n}$ is determined by
$\mu_0 = \alpha_p^{-n} p^{n(k-\frac12)}$ and 
\begin{eqnarray*}
\{\mu_1,...,\mu_{2n}\} = \{\alpha_p p^{-n+1/2},\alpha_p p^{-n+3/2},...,\alpha_p p^{n-1/2} \}
\end{eqnarray*}
up to the action of the Weyl group.
\end{lemma}

Let $\{\alpha_p^{\pm} \}$ and $A_{p,m}$ be the same symbols in Section~\ref{sec:intro}.
We set
the symmetric power $L$-function  of $f$ by
\begin{eqnarray*}
 L(s, \mbox{sym}_m\, f)
  &:=&
 \prod_p \det(1_{m+1} - A_{p,m} \cdot p^{m(k-\frac12)-s})^{-1}.
\end{eqnarray*}
Here, if $m=0$ we set $L(s, \mbox{sym}_0\, f)  := \zeta(s) = \prod_p (1-p^{-s})^{-1}$.

%

In~\cite{Sch} two functions $\alpha(n,m,r)$ and $\beta(n,m,r)$ are introduced.
We quote these definitions from~\cite{Sch}:
Let $ \alpha(r,m,n)$ be the number of possibilities to choose $m$ different numbers
from the set $\{1-2n,3-2n,...,2n-1 \}$ (not observing the order) such that their sum equals $r$.
We define
\begin{eqnarray}\label{id:beta}
 \beta(r,m,n) &:=& \alpha(r,m,n) - \alpha(r,m-2,n).
\end{eqnarray}
The reader is referred to~\cite{Sch}
for the table of numerical values of $\alpha(r,m,n)$ and $\beta(r,m,n)$ $(n \leq 6)$.

From Lemma~\ref{lem:spin_ms} the following theorem follows.

\begin{theorem}[\cite{Mu,Sch}]
The spinor $L$-function of $F_{2n}$ satisfies the identity
\[
 L(s,F_{2n},\mbox{spin})
  =
 \prod_{m=0}^n \prod_{\begin{smallmatrix}r = -m(2n-m) \\ (step\ 2) \end{smallmatrix}}^{m(2n-m)} 
 L(s-m(k-\frac12)+\frac{r}{2}, \mbox{sym}_{n-m}\, f)^{\beta(r,m,n)}.
\]

\end{theorem}
We remark that for any $n$ we have $\beta(r,0,n) =  1$ if $r = 0$, $0$
otherwise. Hence 
\[
 L(s,F_{2n},\mbox{spin})
 =
 L(s, \mbox{sym}_n\, f) \prod_{m=1}^{n} \prod_{r} L(s-m(k-\frac12)+\frac{r}{2}, \mbox{sym}_{n-m}\, f)^{\beta(r,m,n)} .
\]

\section{The spinor $L$-function of Miyawaki-Ikeda lifts} \label{sec:spinor_Miyawaki}

In this Section~\ref{sec:spinor_Miyawaki},
first we review Miyawaki-Ikeda lifts in Section~\ref{ss:mi_lift}.
And we restrict ourselves to
a special case of Miyawaki-Ikeda lifts, namely we consider the lift
\begin{eqnarray*}
 S_{2k}(\Gamma_1) \otimes S_{k+n}(\Gamma_1) \rightarrow
 S_{k+n}(\Gamma_{2n-1}).
\end{eqnarray*}
From Section~\ref{ss:index_shift} to Section~\ref{ss:gen_maass} we recall
some properties of Fourier-Jacobi coefficients of Duke-Imamoglu-Ibukiyama-Ikeda lifts.
Finally, we shall prove Theorem~\ref{thm:main} in Section~\ref{ss:proof}
and give some examples of spinor $L$-functions in Section~\ref{ss:example}.

\subsection{Miyawaki-Ikeda lifts}\label{ss:mi_lift}
We review Miyawaki-Ikeda lifts shown by Ikeda~\cite{Ik2}.
Let $r$ be an integer such that $r \leq n$.
Let $f \in S_{2k}(\Gamma_1)$ be a normalized Hecke eigenform and let
$g \in S_{k+n}(\Gamma_r)$ be a Hecke eigenform.
Let $F_{2n} \in
S_{k+n}(\Gamma_{2n})$ be the Duke-Imamoglu-Ibukiyama-Ikeda lift of $f$.
For $\tau \in \H_{2n-r}$ we put
\begin{eqnarray*}
 \mathcal{F}_{f,g}(\tau)
 :=
 \int_{\Gamma_r \backslash \H_r}
    F_{2n}\left(\smat{\tau}{0}{0}{\omega} \right) \overline{g(-\overline{\omega})}
            \det\left(\mbox{Im}(\omega)\right)^{k+n} \, d\omega .
\end{eqnarray*}
Here
 $d\omega :=
   \det\left(\mbox{Im}(\omega)\right)^{-r-1} \prod_{1\leq i,j \leq r} d\, \mbox{Re}(\omega_{i,j}) \cdot d\, \mbox{Im}(\omega_{i,j})$.
Then it is not difficult to show that the form $\mathcal{F}_{f,g}$
belongs to $S_{k+n}(\Gamma_{2n-r})$.

In \cite{Ik2} Ikeda has  shown:
 If $\mathcal{F}_{f,g} \not \equiv 0$, then 
 the function $\mathcal{F}_{f,g}$ is a common eigenform for Hecke operators
 for
 $\mathcal{H}(\Gamma_{2n-r}, \mbox{Sp}_{2n-r}(\Qq))$.
 Moreover the standard $L$-function of $\mathcal{F}_{f,g}$ satisfies the identity
\begin{eqnarray*}
 L(s,\mathcal{F}_{f,g}, \mbox{st})
 &=&
 L(s,g,\mbox{st}) \prod_{i=1}^{2n-2r} L(s+k+n-r-i,f) .
\end{eqnarray*}

In the present article we only consider the case $r=1$, but $n$ is any positive-integer.
In the case of $r=1$, the above identity of $L(s,\mathcal{F}_{f,g},\mbox{st})$ is also shown in~\cite{MaassRe}
by using certain generalized Maass relations.

\subsection{Index-shift operators for Jacobi forms}\label{ss:index_shift}
For $M = \smat{A}{B}{C}{D} \in \mbox{GSp}_{n}^+(\R)$ we define
\begin{eqnarray*}
 \rho(M) := \left( \begin{smallmatrix} 
                                             A &  & B & \\
                                                 & \nu(M) & & \\
                                             C &  & D & \\ 
                                                 & & & 1
                             \end{smallmatrix} \right)
                             \in \mbox{GSp}_{n+1}^+(\R).
\end{eqnarray*}

For a function $\phi$ on $\mathfrak{H}_{n} \times \C^{n}$ and for
$M = \smat{A}{B}{C}{D} \in \mbox{GSp}_{n}^+(\R)$ we define
\begin{eqnarray*}
 (\phi|_{k,m}\rho(M))(\tau,z)
  &:=&
  \det(C\tau + D)^{-k}\, e\! \left( m\, \nu(M)\, {^t z}(C\tau+D)^{-1}C z\right)\, \phi(M\cdot(\tau,z)) ,
\end{eqnarray*}
where $e(x) := e^{2\pi i x}$ and $M\cdot (\tau,z) := (M\cdot \tau, \nu(M) {^t (C\tau+D)}^{-1}z)$.



For any positive integer $N$ and for $\phi \in J_{k,m}(\Gamma_n)$,
we define two operators $V_{n}(N)$ and $U_{n}(N)$ :
\begin{eqnarray*}
 \phi| V_{n}(N)
 &:=&
 N^{(n+1)k/2-<n>}\sum_{M \in \Gamma_{n}\backslash S_{n}(N)}
                                   \phi|_{k,m} \rho(M), \\
 \phi| U_{n}(N)
 &:=&
 N^{(n+1)k-2<n>} \phi|_{k,m} \rho(N \cdot 1_{2n})\\
 &=&
 N^{-(n+1)n + k} \phi(\tau,Nz).
\end{eqnarray*}
We remark that $\phi| V_{n}(N) \in J_{k,mN}(\Gamma_n)$ and $\phi|U_{n}(N) \in J_{k,mN^2}(\Gamma_n)$
(cf.~\cite[Proposition 4.1]{Ya}.)
And if $\phi \in J_{k,m}^{cusp}(\Gamma_n)$, then $\phi|V_{n}(N)$ and $\phi|U_{n}(N)$
are Jacobi cusp forms.


\subsection{Restriction map}

We define a linear map
  $\mathbb{W}_0 : J_{k,m}^{cusp}(\Gamma_{n}) \rightarrow S_{k}(\Gamma_{n})$
by
\begin{eqnarray*}
(\mathbb{W}_0(\phi))(\tau) := \phi(\tau,0)
\end{eqnarray*}
 for
  $\phi \in J_{k,m}^{cusp}(\Gamma_{n})$ and for $\tau \in \H_n$.

By straightforward calculation we obtain the following lemma.
\begin{lemma}\label{lem:VT}
 For any $\phi \in J_{k,m}^{cusp}(\Gamma_{n})$
 we have
 \begin{eqnarray*}
  \mathbb{W}_0(\phi|V_{n}(N)) = N^{-(n-1)k/2} (\mathbb{W}_0\phi)|T_{n}(N)
 \end{eqnarray*}
 and
 \begin{eqnarray*}
  \mathbb{W}_0(\phi|U_{n}(N)) = N^{-(n+1)n + k} \, \mathbb{W}_0 \phi.
 \end{eqnarray*}
\end{lemma}

\subsection{A generalized Maass relation}\label{ss:gen_maass}

Let
 $F_{2n}\left(\smat{\tau}{z}{{^t z}}{\omega} \right) = \sum_m \Phi_m^{F_{2n}}(\tau,z) e^{2\pi i m \omega}$
be the Fourier-Jacobi expansion of $F_{2n}$, where
$\tau \in \H_{2n-1}$, $\omega \in \H_1$ and $z \in M_{2n-1,1}(\C)$.
Here $F_{2n}$ is the Duke-Imamoglu-Ibukiyama-Ikeda lift of $f \in S_{2k}(\Gamma_1)$.

We define the operator $D_{2n-1}(m,\{\alpha_p^{\pm} \})$
through the  formal Dirichlet-series

\begin{eqnarray*}
 &&
 \sum_{m > 0} \frac{D_{2n-1}(m,\{\alpha_p^{\pm}\})}{m^s}
 := \\
 &&
 \qquad
 \prod_{p} (1 - G_p(\alpha_p) V_{2n-1}(p) p^{\frac12(n-1)(n+2)-s} 
    +U_{2n-1}(p) p^{2n(2n-1)-1-2s})^{-1},
\end{eqnarray*}
where
\begin{eqnarray*}
 G_p(\alpha_p) := G_{p,2n-1}(\alpha_p) = 
 \begin{cases}
  \displaystyle{\prod_{i=1}^{n-1}
                  \left\{
                    \left( 1+ \alpha_p p^{(1-2i)/2} \right) \left( 1 + \alpha_p^{-1} p^{(1-2i)/2} \right)
                  \right\}^{-1}}
  & \mbox{ if } n > 1, \\
  1 & \mbox{ if } n = 1.
 \end{cases}
\end{eqnarray*}
Here the set of complex numbers $\{\alpha_p^{\pm}\}$ is the same notation in
Section~\ref{sec:intro},~\ref{sec:Ikeda}.
Because any two operators in $\{V_{2n-1}(p)\}_{p} \cup \{U_{2n-1}(p)\}_{p}$
are compatible, the above definition of $D_{2n-1}(m,\{\alpha_p^{\pm}\})$ is well-defined.

In particular, for any prime $p$ we have
\begin{eqnarray*}
 D_{2n-1}(p,\{\alpha_p^{\pm}\})
 &=&
 G_p(\alpha_p)\, p^{\frac12(n-1)(n+2)}\, V_{2n-1}(p) .
\end{eqnarray*}

In~\cite{FJlift} we obtained the following relation.
\begin{lemma}[\cite{FJlift}]\label{lem:FJ1}
For any positive integer $m$ we have
\begin{eqnarray*}
 \Phi_m^{F_{2n}} = \Phi_1^{F_{2n}}|D_{2n-1}(m,\{\alpha_p^{\pm}\}).
\end{eqnarray*}
\end{lemma}
We remark that the relation in Lemma~\ref{lem:FJ1} was originally given
by Yamazaki~\cite{Ya} in the case of Siegel-Eisenstein series
of arbitrary degrees.
Lemma~\ref{lem:FJ1} is a certain kind of generalization of the Maass relation.

In~\cite{KK} Katsurada and Kawamura showed the following lemma.
\begin{lemma}[\cite{KK}]\label{lem:kk}
For any $\phi \in J_{k,l}(\Gamma_{2n-1})$
and for any set of complex numbers $\{c_p^{\pm}\}$
we have
\begin{eqnarray*}
 &&
 \phi| D_{2n-1}(m,\{c_p^{\pm}\}) \cdot D_{2n-1}(m',\{c_p^{\pm}\}) \\
 &=&
 \sum_{d| gcd(m,m')} d^{2n(2n-1)-1} \phi|D_{2n-1}\!\left(\frac{mm'}{d^2},\{c_p^{\pm} \}\right) \cdot U_{2n-1}(d) .
\end{eqnarray*}
\end{lemma}

From Lemma~\ref{lem:FJ1} and Lemma~\ref{lem:kk}, we have the following corollary.
\begin{cor}\label{cor:maass}
For any positive integers $m$ and $m'$, we have
\begin{eqnarray*}
 \Phi_m^{F_{2n}}|D_{2n-1}(m',\{\alpha_p^{\pm}\})
 &=&
 \sum_{d| gcd(m,m')} d^{2n(2n-1)-1} \Phi_{\frac{mm'}{d^2}}^{F_{2n}} | U_{2n-1}(d) .
\end{eqnarray*}
\end{cor}


\subsection{Proof of Theorem~\ref{thm:main}}\label{ss:proof}
In this subsection we shall prove Theorem~\ref{thm:main}.

Let $f$ and $g$ be the same notation in the introduction,
and let $\mathcal{F}_{f,g}$ be the Miyawaki-Ikeda lift of $(f,g)$.
The other symbols like
$\lambda_g(p)$, $\{\alpha_p^{\pm}\}$, $\{\beta_p^{\pm} \}$, $F_{2n}$ and $\Phi_{m}^{F_{2n}}$
are the same as before.

First we shall show the following lemma.
\begin{lemma}\label{lem:Tp}
If $\mathcal{F}_{f,g} \not \equiv 0$, then $\mathcal{F}_{f,g}$ is an eigenform for 
the Hecke operator $T_{2n-1}(p)$ for any prime $p$ with the eigenvalue
\[
\lambda_g(p) p^{-(n-1)(n+2)/2}p^{(n-1)(k+n)} \prod_{1\leq i \leq n-1}
\left\{(1+\alpha_p p^{(1-2i)/2})(1+\alpha_p^{-1} p^{(1-2i)/2}) \right\}.
\]
\end{lemma}
\begin{proof}
We put
\begin{eqnarray*}
 C_1 := \displaystyle{p^{-(n-1)(n+2)/2}p^{(n-1)(k+n)} \prod_{1\leq i \leq n-1}
\left\{(1+\alpha_p p^{(1-2i)/2})(1+\alpha_p^{-1} p^{(1-2i)/2}) \right\}}.
\end{eqnarray*}
Then, by using Lemma~\ref{lem:VT} and Corollary~\ref{cor:maass}, we have
\begin{eqnarray*}
 \mathbb{W}_0(\Phi_m^{F_{2n}}) | T_{2n-1}(p) 
 &=&
  p^{(n-1)(k+n)}\, \mathbb{W}_0(\Phi_m^{F_{2n}} | V_{2n-1}(p) )
 \\
 &=&
   G_p(\alpha_p)^{-1} p^{(n-1)(k+n)-\frac12(n-1)(n+2)}\,
 \mathbb{W}_0(\Phi_m^{F_{2n}} | D_{2n-1}(p,\{\alpha_p^{\pm} \}) )
 \\
  &=&
  C_1\,
  \mathbb{W}_0(\Phi_{mp}^{F_{2n}}  + p^{2n(2n-1)-1}\Phi_{m/p}^{F_{2n}}|U_{2n-1}(p)) 
  \\
  &=&
  C_1\, 
\left( \mathbb{W}_0(\Phi_{mp}^{F_{2n}}) + p^{k+n-1} \mathbb{W}_0(\Phi_{m/p}^{F_{2n}})  \right) .
\end{eqnarray*}
Here we regard $\Phi_{m/p}^{F_{2n}}$ as identically $0$ if $m$ is not divisible by $p$.

Hence
\begin{eqnarray*}
 F_{2n}\left(\smat{\tau}{0}{0}{\omega} \right) | T_{2n-1}(p)
 &=&
 \sum_{m > 0}\left(\mathbb{W}_0(\Phi_m^{F_{2n}}) | T_{2n-1}(p) \right) e^{2\pi i m \omega} \\
 &=&
 C_1 \sum_{m>0} \left\{ \mathbb{W}_0(\Phi_{mp}^{F_{2n}}) + p^{k+n-1} \mathbb{W}_0(\Phi_{m/p}^{F_{2n}})  \right\} e^{2\pi i m \omega} \\
 &=&
 C_1 F_{2n}\left(\smat{\tau}{0}{0}{\omega} \right) | T_1(p) ,
\end{eqnarray*}
where
the Hecke operator $T_{2n-1}(p)$ acts on $F_{2n}\left(\smat{\tau}{0}{0}{\omega} \right)$
with respect to the variable $\tau \in \H_{2n-1}$,
while the Hecke operator $T_1(p)$ acts on $F_{2n}\left(\smat{\tau}{0}{0}{\omega} \right)$
with respect to the variable $\omega \in \H_1$.

Because
\begin{eqnarray*}
 \mathcal{F}_{f,g}(\tau)
 =
 \int_{\Gamma_1\backslash \H_1} F_{2n}\left(\smat{\tau}{0}{0}{\omega} \right) \overline{g(\omega)} \mbox{Im}(\omega)^{k+n} \, d\omega ,
\end{eqnarray*}
we  conclude
\begin{eqnarray*}
 \mathcal{F}_{f,g}|T_{2n-1}(p)
 &=&
 \int_{\Gamma_1\backslash \H_1} 
      \left(F_{2n}\left(\smat{\tau}{0}{0}{\omega} \right)|T_{2n-1}(p) \right) \overline{g(\omega)} \mbox{Im}(\omega)^{k+n-2} \, d\omega
  \\
 &=&
      C_1 \int_{\Gamma_1\backslash \H_1}
                 F_{2n}\left(\smat{\tau}{0}{0}{\omega} \right) \overline{g(\omega)|T_1(p)} \mbox{Im}(\omega)^{k+n-2} \, d\omega
  \\
 &=&
 \lambda_g(p)\, C_1\, \mathcal{F}_{f,g}.
\end{eqnarray*}
Thus this lemma follows.
\end{proof}


Let $\{\mu_0,\mu_1,...,\mu_{2n-1} \}$ be the Satake parameters of
$\mathcal{F}_{f,g}$.
Here
\[
\mu_0^2 \prod_{i=1}^{2n-1} \mu_i = p^{(2n-1)(k+n)-<2n-1>}.
\]

We now calculate the Satake parameters of $\mathcal{F}_{f,g}$
by  using Lemma~\ref{lem:Tp}.
\begin{lemma}\label{lem:Ffg_eigenvalue}
For $n \geq 2$ we obtain 
\begin{eqnarray*}
\mu_0 = \alpha_p^{-n+1} \beta_p^{-1} p^{(n-1)(k-\frac12) +(k+n-1)/2}
\end{eqnarray*}
and
\begin{eqnarray*}
 \{ \mu_1,...,\mu_{2n-1} \} = 
 \{ \alpha_p p^{-n+3/2}, \alpha_p p^{-n+5/2},...,
 \alpha_p p^{n-3/2}, \beta_p^2 \}
\end{eqnarray*}
up to the action of the Weyl group.
\end{lemma}
\begin{proof}
We remark
\begin{eqnarray*}
 L(s,g,\mbox{st})
 &=&
 \zeta(s) \prod_p \left\{ (1-\beta_p^2 p^{-s}) (1 - \beta_p^{-2}p^{-s}) \right\}^{-1}.
\end{eqnarray*}

Hence, due to the identity of the standard $L$-function of $\mathcal{F}_{f,g}$ in Section~\ref{ss:mi_lift},
we can take
\begin{eqnarray*}
\{ \mu_1,...,\mu_{2n-1} \} = 
\{ \alpha_p p^{-n+3/2}, \alpha_p p^{-n+5/2},..., \alpha_p p^{n-3/2}, \beta_p^2 \}.
\end{eqnarray*} 

First we assume $\beta_p^2 \neq -1$.
It is known that the eigenvalue of $\mathcal{F}_{f,g}$ for $T_{2n-1}(p)$ is
$\mu_0 \displaystyle{\prod_{i=1}^{2n-1}(1+\mu_i)}$
(cf. \cite[p.257 Hilfsatz 3.14 (b)]{Fr}.) And $\lambda_g(p) = \beta_p^{-1}p^{\frac{k+n-1}{2}} (1 + \beta_p^2)$.
Thus, by virtue of Lemma~\ref{lem:Tp}, we obtain
$\mu_0 = \alpha_p^{-n+1} \beta_p^{-1} p^{(n-1)(k-\frac12) +(k+n-1)/2}$.

If $\beta_p^2 = -1$, then
without loss of generality we can set $\mu_{2n-1} = \beta_p^2 = -1$.
Because $\{\mu_0,\mu_1,...,\mu_{2n-2},-1 \}$ and $\{-\mu_0,\mu_1,...,\mu_{2n-2},-1 \}$
belong to the same equivalent class with respect to the action of the Weyl group,
and because of the identity
\begin{eqnarray*}
\mu_0^2
      = p^{(2n-1)(k+n)-<2n-1>} \mu_1^{-1} \cdots \mu_{2n-1}^{-1}
      = \alpha_p^{-2(n-1)} \beta_p^{-2} p^{(n-1)(2k-1)+k+n-1},
\end{eqnarray*}
we can take $\mu_0 = \alpha_p^{-n+1} \beta_p^{-1} p^{(n-1)(k-\frac12) +(k+n-1)/2}$.
\end{proof}

As a consequence of Lemma~\ref{lem:Tp} and Lemma~\ref{lem:Ffg_eigenvalue},
we obtain Theorem~\ref{thm:main}.


\subsection{Examples of spinor $L$-functions of Miyawaki-Ikeda lifts}\label{ss:example}
In this subsection we give some examples.
We assume $\mathcal{F}_{f,g} \not \equiv 0$ for each case.

\vspace{5mm}

\noindent
\textbf{Degree 3}.

If $n=2$, then $f \in S_{2k}(\Gamma_1)$, $g \in S_{k+2}(\Gamma_1)$ and $\mathcal{F}_{f,g} \in S_{k+2}(\Gamma_3)$. We have
\begin{eqnarray*}
 L(s,\mathcal{F}_{f,g},\mbox{spin})
 &=&
  L(s-k,g)\, L(s-k+1,g)\, L(s,g\otimes f).
\end{eqnarray*}
Here we set $L(s,g\otimes f) := L(s,g\otimes \mbox{sym}_1\, f)$.

We remark that it has been shown in~\cite{He} that
if the pullback of the Duke-Imamoglu-Ibukiyama-Ikeda lift $F_4\left( \smat{\tau}{0}{0}{\omega} \right)$ is not identically zero,
then there exists Hecke eigenforms $H \in S_{k+2}(\Gamma_3)$ and $g \in S_{k+2}(\Gamma_1)$
such that $L(s,H,\mbox{spin}) = L(s-k,g) L(s-k+1,g) L(s, g \otimes f )$,
where $f \in S_{2k}(\Gamma_1)$ is the preimage of the Duke-Imamoglu-Ibukiyama-Ikeda lift $F_4 \in S_{k+2}(\Gamma_4)$.

\vspace{5mm}

\noindent
\textbf{Degree 5}.

If $n=3$,
then $f \in S_{2k}(\Gamma_1)$, $g \in S_{k+3}(\Gamma_1)$ and
$\mathcal{F}_{f,g} \in S_{k+3}(\Gamma_5)$.
We have
\begin{eqnarray*}
 L(s,\mathcal{F}_{f,g},\mbox{spin})
 &=&
 L(s,g\otimes \mbox{sym}_2\, f)
 \prod_{i=-1}^{2} L(s-k+i, g\otimes f)
 \prod_{i=-1}^{3} L(s-2k+i, g).
\end{eqnarray*}

\vspace{5mm}

\noindent
\textbf{Degree 7}.

If $n=4$, then $f \in S_{2k}(\Gamma_1)$, $g \in S_{k+4}(\Gamma_1)$ and
$\mathcal{F}_{f,g} \in S_{k+4}(\Gamma_7)$.
We have
\begin{eqnarray*}
 L(s,\mathcal{F}_{f,g},\mbox{spin})
 &=&
 L(s,g\otimes \mbox{sym}_3\, f)
 \prod_{i=-2}^{3} L(s-k+i,g\otimes \mbox{sym}_2\, f) \\
 && \times \prod_{i=-3}^{5} L(s-2k+i,g\otimes f)^{\varepsilon_i}
                   \prod_{i=-3}^{6} L(s-3k+i,g)^{\varepsilon'_i},
\end{eqnarray*}
where
$\varepsilon_i = \begin{cases} 1 & \mbox{if } i = -3,-2,4,5, \\
                                                        2 & \mbox{if } i = -1,0,1,2,3
                              \end{cases}$
and
$\varepsilon'_i = \begin{cases} 1 & \mbox{if } i = -3,-2,-1,4,5,6, \\
                                                  2 & \mbox{if } i = 0,1,2,3.
                               \end{cases}$

\vspace{1cm}

\noindent
Department of Mathematics, Joetsu University of Education,\\
1 Yamayashikimachi, Joetsu, Niigata 943-8512, JAPAN\\
e-mail hayasida@juen.ac.jp

\end{document}